\documentclass[11pt]{article}

\usepackage[dvips]{graphicx}
\usepackage{color}
\usepackage{amsmath}
\usepackage{amsfonts}
\usepackage{amssymb}
\usepackage{amsthm}
\usepackage{newlfont}
\usepackage{epsfig}
\usepackage{multirow}
\usepackage{setspace}
\usepackage{hyperref}
\usepackage{cite}

\begin{document}

\newtheorem{assumption}{Assumption}[section]
\newtheorem{definition}{Definition}[section]
\newtheorem{lemma}{Lemma}[section]
\newtheorem{proposition}{Proposition}[section]
\newtheorem{theorem}{Theorem}[section]
\newtheorem{corollary}{Corollary}[section]
\newtheorem{remark}{Remark}[section]

\title{Weak Dynamic Non-Emptiability and Persistence of Chemical Kinetics Systems}

\author{Matthew D. Johnston\thanks{Supported by a Natural Sciences and Engineering Research Council of Canada Post-Graduate Scholarship} and David Siegel\thanks{Supported by a Natural Sciences and Engineering Research Council of Canada Research Grant \newline \textbf{Keywords:} chemical kinetics; dynamically non-emptiable; siphons; complex balanced; facets \newline \textbf{AMS
  Subject Classifications:} 80A30, 34D20, 37C75.}
  \vspace*{.2in} \\
Department of Applied Mathematics, University of Waterloo, \\
Waterloo, Ontario, Canada N2L 3G1 }
\date{}
\maketitle

\bigskip

\begin{abstract}

A frequently desirable characteristic of chemical kinetics systems is that of persistence, the property that if all the species are initially present then none of them may tend toward extinction. It is known that solutions of deterministically modelled mass-action systems may only approach portions of the boundary of the positive orthant which correspond to semi-locking sets (alternatively called siphons). Consequently, most recent work on persistence of these systems has been focused on these sets.

In this paper, we focus on a result which states that, for a conservative mass-action system, persistence holds if every critical semi-locking set is dynamically non-emptiable and the system contains no nested locking sets. We will generalize this result by introducing the notion of a weakly dynamically non-emptiable semi-locking set and making novel use of the well-known Farkas' Lemma. We will also connect this result to known results regarding complex balanced systems and systems with facets.

\end{abstract}

\bigskip

\section{Introduction}

An elementary chemical reaction is given by a set of reactant species reacting at a prescribed rate to form a set of product species. Under appropriate assumptions (well-mixing, constant external conditions, large number of reacting molecules, etc.) such systems can be modelled deterministically by an autonomous set of ordinary differential equations over continuous variables representing the concentrations of the reactant species. The resulting mathematical models have a long history and enjoy applications in fields such as systems biology, industrial chemistry, atmospherics, etc. \cite{B-B1,F1,F2,H,H-J1}

One topic which has gained significant attention recently has been that of persistence \cite{A,A-S,A3,C-D-S-S,S-C,S-M}. A chemical kinetics system is said to be \emph{persistent} if no initially present species may tend toward extinction. In general, determining whether a system is persistent can be difficult, but it has been significantly simplified by recent work. In particular, in \cite{A3} the authors show that the boundary of the positive orthant can be divided into subsets $L_I$ (roughly faces of $\mathbb{R}_{>0}^m$) such that trajectories may only approach $L_I$ if $I \subseteq \left\{ 1, \ldots, m \right\}$ is a \emph{siphon} (called \emph{semi-locking sets} in this paper). Consequently, in order to determine persistence it is sufficient to consider only the behaviour near sets $L_I$ corresponding to sets $I$ which are semi-locking sets.

In the same paper, the authors introduce a condition on the sets $I$ called \emph{dynamical non-emptiability}. This condition corresponds to a dominance ordering of the reactant monomials which compose the differential equations. Together with some technical assumptions, they prove that the dynamical non-emptiability of every critical semi-locking set is sufficient to guarantee persistence. While certainly powerful, the result is limited by its restrictive assumptions, which include that the system be conservative and have no nested critical deadlocks (Theorem 4, \cite{A3}). In this paper, we extend this result by introducing the notion of \emph{weak dynamical non-emptiability}. Our methodology differs significantly from that used in \cite{A3} and will allow us to relax several assumptions while maintaining persistence. Our methodology depends crucially on a novel use of the famous Farkas' Lemma \cite{F}.

For many systems, proving persistence is tantamount to proving the global asymptotic stability of some positive equilibrium state, a characteristic which is typically highly desirable. One such class of systems are the \emph{complex balanced systems} first considered in \cite{F1,H,H-J1}. In \cite{H-J1}, the authors proved that relative to each stoichiometric compatibility class, the invariant spaces in which solutions are restricted, there is precisely one positive equilibrium state and that this state is locally asymptotically stable.

Despite significant empirical support, however, the hypothesis that this equilibrium state is in fact \emph{globally} asymptotically stable relative to its compatibility class remains unproven. In the literature, this hypothesis has been termed the \emph{Global Attractor Conjecture} and a significant amount of work has been done attempting to confirm it \cite{A,A3,C-D-S-S,S-C,S-M}. Of particular interest to us is the work of \cite{S-M}, where the authors prove that the $\omega$-limit set consists of either this unique positive equilibrium state or of some complex balanced equilibria on the boundary of the positive orthant. Consequently, for complex balanced systems, persistence is a sufficient condition for the global asymptotic stability of this equilibrium state relative to its stoichiometric compatibility class. It follows that the work of this paper affirms the \emph{Global Attractor Conjecture} for chemical kinetics systems satisfying the relevant assumptions which are also complex balanced.

Persistence is also considered in \cite{A-S} where the authors prove that mass-action systems are persistent if every semi-locking set $I$ corresponds to a facet of the compatibility class (i.e. a facet is a codimension one face of the compatibility class). The authors connect their work with \cite{A3} by showing that sets $I$ corresponding to facets are dynamically non-emptiable under some assumptions of the system. We will prove that there is no restriction for weak dynamical non-emptiability; that, in fact, every set $I$ corresponding to a facet is weakly dynamically non-emptiable.

It should be noted that our notation and terminology will often differ from that of \cite{A3}. We have attempted, wherever possible, to adopt terminology specific to chemical kinetics systems (for example, we follow \cite{A} in naming siphons according to their chemical kinetics equivalent of semi-locking sets). Throughout this paper, we will let $\mathbb{R}_{>0}^m$ and $\mathbb{R}_{\geq 0}^m$ denote the $m$-dimensional spaces with all coordinates strictly positive and non-negative, respectively.

\section{Background}

Within the mathematical literature, several distinct ways to represent chemical reactions have been proposed. In this paper, we will follow closely the \emph{reaction}-oriented formulation of \cite{A3}. (For examples of \emph{species}- and \emph{complex}-oriented formulations, see \cite{V-H} and \cite{H-J1}, respectively.)

We will let $\mathcal{A}_j$ denote the \emph{species} or \emph{reactants} of the system and let $\mathcal{S}$ denote the set of distinct species of the system. We define $|\mathcal{S}| = m$. The reactions of the system will be represented as 
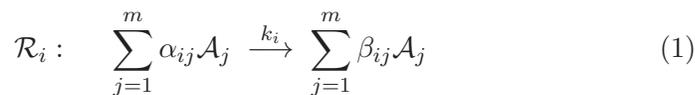
\begin{equation}
\label{reactionnetwork}
\mathcal{R}_i: \; \; \; \; \; \sum_{j=1}^m \alpha_{ij} \mathcal{A}_j \; \stackrel{k_i}{\longrightarrow} \; \sum_{j=1}^m \beta_{ij} \mathcal{A}_j
\end{equation}
where $\alpha_{ij}$, $\beta_{ij} \in \mathbb{Z}_{\geq 0}$ are called the \emph{stoichiometric coefficients} and $k_i > 0$ is the \emph{rate constant} associated with the $i^{th}$ reaction. We will let $\mathcal{R}$ denote the set of reactions of the system and define $|\mathcal{R}|=r$.

\subsection{Mass-Action Kinetics}

It is convenient to collect the net gain or loss of each species as a result of a specified reaction in a matrix $\Gamma \in \mathbb{Z}^{m \times r}$ where the entries of $\Gamma$ are defined according to
\begin{equation}
\label{gamma}
[\Gamma]_{ji} = \beta_{ij}-\alpha_{ij}
\end{equation}
for $i=1, \ldots, r$ and $j = 1, \ldots, m$ (notice the reversal of indices). The entries $[\Gamma]_{ji}$ can be intuitively interpreted as the change in the $j^{th}$ species as a result of each instance of the $i^{th}$ reaction.

We will let $x_j=[\mathcal{A}_j]$ denote the concentration of the $j^{th}$ species and denote by $\mathbf{x} = [ x_1 \; x_2 \; \cdots \; x_m]^T$ the \emph{concentration vector}. We define the \emph{reaction vector} $R(\mathbf{x}) \in \mathbb{R}_{\geq 0}^r$ according to
\begin{equation}
\label{reactionvector}
R_i(\mathbf{x}) = k_i \prod_{j=1}^m x_j^{\alpha_{ij}}
\end{equation}
for $i = 1, \ldots, r$. The terms $R_i(\mathbf{x})$ correspond to the \emph{reaction rates} of chemical reactions under the assumption of mass-action kinetics. (For further discussion on alternatives to mass-action kinetics, see \cite{A3}.)

Since $[\Gamma]_{ji}$ represents the net stoichiometric change in the $j^{th}$ species as a result of the $i^{th}$ reaction and $R_i(\mathbf{x})$ represents the rate of occurrence of the $i^{th}$ reaction, it follows that $[\Gamma]_{ji} R_i(\mathbf{x})$ is the rate of change of the $j^{th}$ species as a result of the $i^{th}$ reaction. It follows that (\ref{reactionnetwork}) is governed by the system of differential equations
\begin{equation}
\label{de}
\frac{d\mathbf{x}}{dt} = \mathbf{f}(\mathbf{x}) = \Gamma R(\mathbf{x}).
\end{equation}

It follows from (\ref{de}) that solutions are not able to wander around freely in $\mathbb{R}^m$ since $\mathbf{f}(\mathbf{x}) \in$ range$(\Gamma)$ for all $\mathbf{x} \in \mathbb{R}_{\geq 0}^m$. Trajectories are in fact restricted to \emph{stoichiometric compatibility classes} (see \cite{H-J1}).

\begin{definition}
\label{stoic}
The \textbf{stoichiometric subspace} for a chemical reaction mechanism (\ref{reactionnetwork}) is the linear subspace $S \subseteq \mathbb{R}^m$ given by
\[S= \mbox{\emph{range}}(\Gamma).\]
\noindent The dimension of the stoichiometric subspace will be denoted by $s =$ \emph{rank}$(\Gamma)$.
\end{definition}

\begin{definition}
The positive \textbf{stoichiometric compatibility class} containing the initial concentration $\mathbf{x}_0 \in \mathbb{R}^m_{>0}$ is the set $\mathsf{C}_{\mathbf{x}_0} = (\mathbf{x}_0 + S) \cap \mathbb{R}^m_{>0}$.
\end{definition}

\begin{proposition} [\cite{H-J1,V-H}]
\label{proposition2}
Let $\mathbf{x}(t)$ be the solution to (\ref{de}) with $\mathbf{x}(0)=\mathbf{x}_0 \in \mathbb{R}^m_{>0}$. Then $\mathbf{x}(t) \in \mathsf{C}_{\mathbf{x}_0}$ for $t \geq 0$.
\end{proposition}

%Note that a solution $\mathbf{x}(t)$ of (\ref{de}) with $\mathbf{x}(0)=\mathbf{x}_0 \in \mathbb{R}^m_{>0}$ may exist only on a finite interval $0 \leq t < %T$, in which case $\mathbf{x}(t) \in \mathsf{C}_{\mathbf{x}_0}$ for $0 \leq t < T$. Throughout this paper we only consider solutions to (\ref{de}) %satisfying %$\mathbf{x}(0)=\mathbf{x}_0 \in \mathbb{R}^m_{>0}$, so that Proposition \ref{proposition2} holds.

A significant amount of literature exists analysing and restricting the behaviour of solutions of (\ref{de}) \cite{A,B-B1,C-D-S-S,F1,H-J1,S-C,S-M,V-H}. We will omit most of this discussion, instead focusing on the analysis conducted in \cite{A3} and the related paper \cite{A}.

\subsection{Conservative Systems and Semi-Locking Sets}

In this section, we introduce several background concepts from \cite{A3} which will be important to us. We have adapted some of their terminology to the more specific framework of chemical kinetics.

We start by defining a \emph{conservative} system.

\begin{definition}
\label{conservative}
A chemical kinetics system is said to be \textbf{conservative} if there exists a $\mathbf{c} \in \mathbb{R}_{>0}^m$ such that
\begin{equation}
\label{0001}
\mathbf{c}^T \Gamma = \mathbf{0}^T.
\end{equation}
We will call any vector $\mathbf{c} \in \mathbb{R}_{>0}^m$ satisfying (\ref{0001}) a \textbf{conservation vector}.
\end{definition}

An important property of conservative systems is that solutions of (\ref{de}) with $\mathbf{x}_0 \in \mathbb{R}_{>0}^m$ are bounded and do not approach the origin. This follows from (\ref{de}) and (\ref{0001}) since these collectively imply $\frac{d}{dt} \left[ \mathbf{c}^T \mathbf{x}(t) \right] = 0$ from which it follows that $\mathbf{c}^T \mathbf{x}(T) = \mathbf{c}^T \mathbf{x}(0) > 0$ for all $T > 0$. Divergence of any species to infinity, however, implies the existence of a sequence $\left\{ t_n \right\}$ such that $\mathbf{c}^T \mathbf{x}(t_n) \to \infty$ as $n \to \infty$, while convergence to the origin implies a sequence such that $\mathbf{c}^T \mathbf{x}(t_n) \to 0$ as $n \to \infty$.

Related to the concept of a \emph{conservation vector} is the concept of a \emph{semi-conservation vector}. (These are called \emph{P-semiflows} in \cite{A3}. We introduce the term \emph{semi-conservation} to emphasize the connection with chemical kinetics.)

\begin{definition}
\label{semiconservative}
Consider a nonempty set $I \subseteq \left\{ 1, 2, \ldots, m \right\}$. A chemical kinetics system is said to be \textbf{conservative with respect to} $I$ if there exists a $\mathbf{c} \in \mathbb{R}_{\geq 0}^m$ satisfying
\begin{equation}
\label{0002}
\mathbf{c} = \left\{ \begin{array}{ll} c_i > 0, \hspace{0.3in} & i \in I \\ c_i = 0, & i \not\in I\end{array} \right.
\end{equation}
so that
\begin{equation}
\label{0003}
\mathbf{c}^T \Gamma = \mathbf{0}^T.
\end{equation}
We will call any vector $\mathbf{c} \in \mathbb{R}_{\geq 0}^m$ satisfying (\ref{0002}) and (\ref{0003}) a \textbf{semi-conservation vector}.
\end{definition}

Many of the results of this paper will depend on restricting our attention to only those components $x_i$ of $\mathbf{x} \in \mathbb{R}^m$ such that $i \in I$ for some set $I \subseteq \left\{1, \ldots, m \right\}$. When $I$ corresponds to only those indices for which $x_i \not= 0$ we will call $I$ the \emph{support} of $\mathbf{x}$.

%In order to make this procedure as unambiguous as possible, we introduce the following %definitions.

%\begin{definition}
%Consider a set $I \subseteq \left\{ 1, \ldots, m \right\}$. Given a vector $\mathbf{x} \in \mathbb{R}^m$, we define the \textbf{restriction of} $\mathbf{x}$ %\textbf{to} $\mathbb{R}^{|I|}$, \emph{res}$(\mathbf{x}) \in \mathbb{R}^{|I|}$, to be the vector with the elements $x_i$, $i \not\in I$, removed whilst %otherwise retaining the order of elements in $\mathbf{x}$. Correspondingly, given a vector $\mathbf{y} \in \mathbb{R}^{|I|}$, we define the \textbf%{extention %of} $\mathbf{y}$ \textbf{to} $\mathbb{R}^{m}$, \emph{ext}$(\mathbf{y}) \in \mathbb{R}^m$, to be the vector with the elements $[\emph{ext}(\mathbf{y})]_i = %0$ %for $i \not\in I$ added whilst otherwise retaining the order of the elements in $\mathbf{y}$.
%\end{definition}
%\noindent The set $I$ typically corresponds to those indices for which $x_i \not= 0$, in which case $I$ is called the \emph{support} of $\mathbf{x}$. It is %clear that $\mathbf{x} = $ext$($res$(\mathbf{x}))$ if and only if $I$ is the support of $\mathbf{x} \in \mathbb{R}^m$.

The following concepts are the central focus of the papers \cite{A} and \cite{A3}. They will be key to the analysis conducted in the rest of this paper.
\begin{definition}
\label{semilockingset}
The nonempty index set $I \subseteq \left\{ 1, 2, \ldots, m \right\}$ is called a \textbf{semi-locking set} (alternatively, a \textbf{siphon}) if, for every reaction where an element $\mathcal{A}_i$, $i \in I$, appears in the product complex, an element from $\mathcal{A}_j$, $j \in I$ appears in the reactant complex. A semi-locking set $I$ is called a \textbf{locking set} (alternatively, a \textbf{deadlock}) if every reaction has an element $\mathcal{A}_i$, $i \in I$, which appears in the reactant complex.

%, for every reaction $\mathcal{R}_i$ where a species $\mathcal{A}_j$, $j \in I$, appears as a product, then a species $\mathcal{A}_j$, $j \in I$, also %appears as a reactant. A semi-locking set is called a \textbf{locking set} (alternatively, a \textbf{deadlock}) if, for every reaction $\mathcal{R}_i$, %there is a $\mathcal{A}_j$, $j \in I$, which appears as a reactant.
\end{definition}

It is worth noting that the element $\mathcal{A}_i$ appearing in the product may be the same as the element $\mathcal{A}_j$ appearing in the reactant. The concepts of siphons and deadlocks are based on the theory of Petri Nets introduced in \cite{Pe} and used extensively in \cite{A3}. We will, however, use the chemical reaction specific terminology of semi-locking and locking sets used in \cite{A}.

There are several intuitive consequences of semi-locking and locking sets. Since any species on the reactant side of a reaction being absent (i.e. $x_i=0$ for this species) implies by (\ref{reactionvector}) that the rate of the reaction is zero, if all the species in a semi-locking set are absent then no species in the set will ever be produced  (i.e. $x_i(t)=0$ for all $t \geq 0$ for all species in the set). In this sense, the set of species is locked in place. In an analogous way, for a locking set all of the species are locked in place.

The authors of \cite{A3} further classify semi-locking sets according to the following definition.

\begin{definition}
\label{critical}
A semi-locking set $I \subseteq \left\{ 1, 2, \ldots, m \right\}$ will be called \textbf{critical} if it does not contain the support of a semi-conservation vector.
\end{definition}

%\begin{definition}
%\label{critical}
%A semi-locking set $I \subseteq \left\{ 1, 2, \ldots, m \right\}$ will be called \textbf{non-critical} if it corresponds to the support and a semi-%conservation vector and \textbf{critical} if it does not.
%\end{definition}

\subsection{Persistence and Faces}
\label{persistencesection}

The primary goal of \cite{A3} is guaranteeing that solutions of (\ref{de}) do not approach $\partial \mathbb{R}_{>0}^m$ for any $\mathbf{x}_0 \in \mathbb{R}_{>0}^m$. In other words, the authors sought to guarantee that if every species is initially present then none of them could not tend toward extinction as a result of the reaction system.

The concept is succinctly stated as follows.

\begin{definition}
\label{persistence}
A chemical reaction network with bounded trajectories is said to be \textbf{persistent} if, for any $\mathbf{x}_0 \in \mathbb{R}_{>0}^m$, it follows that $\omega(\mathbf{x}_0) \cap \partial \mathbb{R}_{>0}^m = \emptyset$.
\end{definition}

Persistence has historical roots in predator-prey modeling in biology but has been frequently applied to chemical reaction mechanisms in recent years. We require trajectories be bounded to avoid ambiguous limiting behaviour.

In order to make the concept of persistence more manageable, we divide the boundaries of $\mathbb{R}_{>0}^m$ and $\mathsf{C}_{\mathbf{x}_0}$ into faces. For technical reasons we will only be interested in the \emph{relative interior} of these faces, which we will define as follows. (These sets are defined similarly in \cite{A}, \cite{A3}, and \cite{C-D-S-S}. In \cite{A-S}, $L_I$ is denoted $Z_I$.)

\begin{definition}
\label{face}
Given a nonempty index set $I \subseteq \left\{ 1, 2, \ldots, m \right\}$, we define the sets $L_I$ and $F_I$ to be
\[\begin{array}{c} \displaystyle{L_I = \left\{ \mathbf{x} \in \mathbb{R}_{\geq 0}^m \; | \; x_i = 0 \mbox{ for } i \in I \mbox{ and } x_i > 0 \mbox{ for } i \not\in I \right\}}\\
\displaystyle{F_I = \left\{ \mathbf{x} \in \overline{\mathsf{C}}_{\mathbf{x}_0} \; | \; x_i = 0 \mbox{ for } i \in I \mbox{ and } x_i > 0 \mbox{ for } i \not\in I \right\}}\end{array}\]
\end{definition}
The set $F_I$ can also be given as $F_I = \overline{\mathsf{C}}_{\mathbf{x}_0} \cap L_I$ or $F_I = (\mathbf{x}_0 + S) \cap L_I$.

We notice that each $\mathbf{x} \in \partial \mathbb{R}_{>0}^m$ can be placed into exactly one set $L_I$ and each $\mathbf{x} \in \partial \mathsf{C}_{\mathbf{x}_0}$ can be placed into exactly one set $F_I$ so that these sets uniquely and completely decompose $\partial \mathbb{R}_{>0}^m$ and $\partial \mathsf{C}_{\mathbf{x}_0}$, respectively. Not every $I$, however, necessarily corresponds to a $F_I$ which contains points, i.e. it is possible that $F_I = \emptyset$. Such sets will be said to be \emph{stoichiometrically unattainable}.

The following result corresponds to Proposition 1 of \cite{A3} and Theorem 2.5 of \cite{A}. This result places strong limitations on where $\omega$-limit points may lie on $\partial \mathbb{R}_{>0}^m$.
\begin{lemma}
\label{lemma461}
Consider a mass-action system and a nonempty index set $I \subseteq \{1, 2, \ldots, m\}$. If there exists a $\mathbf{x}_0 \in \mathbb{R}_{>0}^m$ such that $\omega(\mathbf{x}_0) \cap L_I \not= \emptyset$, then $I$ is a semi-locking set.
\end{lemma}

An important consequence of Lemma \ref{lemma461} is that, in order to prove the persistence of a mass-action system, it is now sufficient to prove only that $\omega(\mathbf{x}_0) \cap L_I = \emptyset$ for all sets $L_I$ corresponding to semi-locking sets $I$. In practice, this is a significant simplification.

\subsection{Dynamical Non-Emptiability}

In Section \ref{persistencesection} we saw that, in order to determine persistence of chemical kinetics systems, it is sufficient to consider behaviour near the sets $L_I$ corresponding to semi-locking sets. To this end, Angeli \emph{et al.} introduced several new concepts, including a notion of partial ordering on reaction rates and two cones: the \emph{feasibility cone} and the \emph{criticality cone}. These ideas culminate in the concept of a semi-locking set being \emph{dynamically non-emptiable}. Everything in this section can be found in \cite{A3}.

We start by defining a partial ordering condition on the reactions of a system.

\begin{definition}
\label{partialorder}
Consider the nonempty index set $I \subseteq \left\{ 1, 2, \ldots, m \right\}$. For $\mathcal{R}_i, \mathcal{R}_j \in \mathcal{R}$, we will say that $\mathcal{R}_i \curlyeqprec_I \mathcal{R}_j$ if $\alpha_{ik} \geq \alpha_{jk}$ for all $k \in I$ and the inequality is strict for at least one $k \in I$.
\end{definition}

Intuitively, the partial ordering condition given in Definition \ref{partialorder} gives us an estimate on the magnitudes of the reaction terms $R_i(\mathbf{x})$ near a set $L_I$. This is made explicit by the following result.

\begin{lemma}[Lemma 4, \cite{A3}]
\label{lemma4}
Consider a mass-action system and let $I \subseteq \left\{ 1, 2, \ldots, m \right\}$ be a semi-locking set. Suppose that $\mathcal{R}_i \curlyeqprec_I \mathcal{R}_j$. Then, for every $\epsilon > 0$, and each compact subset $K$ of $L_I$, there exists a neighbourhood $U$ of $K$ in $\mathbb{R}_{>0}^m$ such that $R_i(\mathbf{x}) \leq \epsilon R_j(\mathbf{x})$ for all $\mathbf{x} \in U$.
\end{lemma}

The concept of \emph{dynamical non-emptiability} depends on two cones, the \emph{feasibility cone} and \emph{criticality cone}, which are defined as follows.

\begin{definition}
\label{feasibilitycone}
The \textbf{feasibility cone} is defined to be
\[\mathcal{F}_\epsilon(I) = \left\{ \mathbf{v} \in \mathbb{R}_{\geq 0}^r \; | \; v_i \leq \epsilon v_j, \; \forall \; \mathcal{R}_i,\mathcal{R}_j \in \mathcal{R} \mbox{ such that } \mathcal{R}_i \curlyeqprec_I \mathcal{R}_j \right\}\]
where $\epsilon > 0$.
\end{definition}
%\noindent The feasibility cone gives us a representation in $r$-dimensional reaction space of the relative relevance of the reactions near the set $L_I$.

\begin{definition}
\label{criticalitycone}
The \textbf{criticality cone} is defined to be
\[\mathcal{C}(I) = \left\{ \mathbf{v} \in \mathbb{R}_{\geq 0}^r \; | \; [\Gamma \mathbf{v}]_k \leq 0, \; \forall \; k \in I \right\}.\]
\end{definition}
%\noindent Since we intuitively expect that any trajectory satisfying $\mathbf{x}_0 \in \mathbb{R}_{>0}^m$ and (\ref{de}) which approaches $L_I$ entails a %decrease in every $x_i$ such that $i \in I$, the criticality cone gives us a representation in $r$-dimensional reaction space of the relative weights of the %reactions required for such a possibility.

We can now define \emph{dynamical non-emptiability}, which is one of the major concepts introduced in \cite{A3}.

\begin{definition}
\label{dynamicallynonemptiable}
A critical semi-locking set is said to be \textbf{dynamically non-emptiable} if there exists an $\epsilon > 0$ such that
\[\mathcal{F}_\epsilon(I) \cap \mathcal{C}(I) = \left\{ \mathbf{0} \right\}.\]
\end{definition}

\subsection{Results of Angeli, De Leenheer, and Sontag}
\label{resultssection}

The following two persistence results are proved in \cite{A3}. They are the basis of the original work contained in Section \ref{originalresultssection}. (Two sets $I_1$ and $I_2$ are \emph{nested} if $I_1 \subset I_2$ or $I_2 \subset I_1$.)

\begin{theorem}[Theorem 2, \cite{A3}]
\label{sontag1}
Consider a chemical reaction network satisfying the following assumptions:
\begin{enumerate}
\item
the system is conservative;
\item
every semi-locking set contains the support of a semi-conservation vector.
\end{enumerate}
Then the system is persistent.
\end{theorem}

\begin{theorem}[Theorem 4, \cite{A3}]
\label{sontag2}
Consider a conservative mass-action system satisfying the following assumptions:
\begin{enumerate}
\item
all of its critical semi-locking sets are dynamically non-emptiable;
\item
there are no nested distinct critical locking sets.
\end{enumerate}
Then the system is persistent.
\end{theorem}

\section{Original Results}
\label{originalresultssection}

In this section, we generalize the results contained in Section \ref{resultssection}. We start by presenting some necessary background material not presented in \cite{A3}.

The following result is a version of the well-known Farkas' Lemma and should be contrasted with Lemma 5 of \cite{A3} as it will be used in similar fashion. (This formulation of the result follows from the statement of Farkas' Lemma given in \cite{R}, taking $\mathbf{a}_0 \in \mathbb{R}_{\geq 0}^m \setminus \left\{ \mathbf{0} \right\}$.)

\begin{lemma}[Farkas' Lemma, \cite{F}]
\label{farkaslemma}
Consider $A \in \mathbb{R}^{m \times n}$. Then exactly one of the following two conditions is true:
\begin{enumerate}
\item
There exists $\mathbf{x} \in \mathbb{R}^n_{\geq 0}$, $\mathbf{x} \not\in \emph{ker}(A)$, such that $A \mathbf{x} \leq \mathbf{0}$.
\item
There exists $\mathbf{y} \in \mathbb{R}^m_{> 0}$ such that $A^T \mathbf{y} \geq \mathbf{0}$.
\end{enumerate}
\end{lemma}

The following persistence result can be found in \cite{S-J}. We do not prove it here.

\begin{theorem}[Theorem 3.13, \cite{S-J}]
\label{bigtheorem}
Consider a general mass-action system with bounded solutions. Suppose that for every $L_I$ corresponding to a semi-locking set $I$ there exists an $\alpha$ satisfying
\begin{equation}
\label{condition1}
\alpha = \left\{ \begin{array}{l} \alpha_i < 0, \mbox{ for } i \in I \\ \alpha_i = 0, \mbox{ for } i \not\in I \end{array} \right.
\end{equation}
and the following property: for every compact subset $K$ of $L_I$, there exists a neighbourhood $U$ of $K$ in $\mathbb{R}_{\geq 0}^m$ such that
\begin{equation}
\label{condition2}
\langle \alpha, \dot{\mathbf{x}} \rangle \leq 0 \mbox{ for all } \mathbf{x} \in U.
\end{equation}
Then the system is persistent.
\end{theorem}

\subsection{Weak Dynamical Non-Emptiability}

In this section, we extend the notion of dynamical non-emptiability by modifying the feasibility cone and introducing a kernel condition. We also reformulate the conditions required for inclusion in the feasibility and criticality cones as matrix conditions which will allow us to prove the main result of the paper (Theorem \ref{sontag4}).

Our notion of dynamical non-emptiability depends on the selection of a set $J \subseteq \mathcal{R}_I$ where
\begin{equation}
\label{J}
\mathcal{R}_I = \left\{ (i,j) \in \left\{ 1, \ldots, r \right\} \times \left\{ 1, \ldots, r \right\} \; | \; \mathcal{R}_i \curlyeqprec_I \mathcal{R}_j \right\}.
\end{equation}
The key modification here is that we do not necessarily need to consider \emph{all} pairs $(i,j)$ satisfying $\mathcal{R}_i \curlyeqprec_I \mathcal{R}_j$; often it is sufficient to consider a strict subset of these pairs of reactions. In Section \ref{examplessection}, we will see that this modification allow us to encompass more chemical kinetics systems than we would be able to otherwise.

We will need the following two concepts.

\begin{definition}
\label{modifiedfeasibilitycone}
We define the \textbf{feasibility cone relative to} $J$ to be
\[\mathcal{F}_\epsilon(J) = \left\{ \mathbf{v} \in \mathbb{R}_{\geq 0}^r \; | \; v_i \leq \epsilon v_j, \mbox{ for all } (i,j) \in J \right\}\]
where $\epsilon > 0$.
\end{definition}

\begin{definition}
\label{kernel}
We define the \textbf{kernel of} $I$ \textbf{and} $J$ to be
\[\begin{split} \emph{ker}(I,J,\epsilon) & = \left\{ \mathbf{v} \in \mathbb{R}_{\geq 0}^r \; | \; [\Gamma \mathbf{v}]_k = 0, \mbox{ for all } k \in I \right. \\ & \hspace{0.86in} \left. \mbox{ and } v_i = \epsilon v_j, \mbox{ for all } (i,j) \in J \right\} \end{split} \]
where $\epsilon > 0$.
\end{definition}

The following notion of dynamical non-emptiability is our own. It is more general than that contained in \cite{A3} in that it makes use of the freedom to select an appropriate $J \subseteq \mathcal{R}_I$ and broadens the inclusion principle to a kernel condition. (It is clear that the standard notion of dynamical non-emptiability is included as a special case of the following by taking $J = \mathcal{R}_I$ and recognizing that $\left\{ \mathbf{0} \right\} \subseteq$ ker$(I,J,\epsilon)$.)

\begin{definition}
\label{weaklydynamicallynonemptiable}
A critical semi-locking set $I$ is said to be \textbf{weakly dynamically non-emptiable} if there exists an $\epsilon > 0$ and a $J$ satisfying (\ref{J}) such that
\[\mathcal{C}(I) \cap \mathcal{F}_\epsilon(J) \subseteq \emph{ker}(I,J,\epsilon).\]
\end{definition}

In order to relate the above conditions to Farkas' Lemma (Theorem \ref{farkaslemma}) we restate them as matrix conditions. We let $n_I = |I|$ and $n_J = |J|$. We define $\Gamma_I \in \mathbb{R}^{n_I \times r}$ to be the matrix $\Gamma$ with the rows $\Gamma_{k \cdot}$, $k \not\in I$, removed. We define $\Gamma_J \in \mathbb{R}^{n_J \times r}$ to be the matrix where each row corresponds to a specific condition $\mathcal{R}_i \curlyeqprec_I \mathcal{R}_j$, $(i,j) \in J$, so that in that row there is a one in the $i^{th}$ column, a $-\epsilon$ in the $j^{th}$ column, and zeroes elsewhere. Lastly, we define $\tilde{\Gamma} \in \mathbb{R}^{(n_I+n_J) \times r}$ to be
\[\tilde{\Gamma} = \left[ \begin{array}{c} \Gamma_I \\ \Gamma_J \end{array} \right].\]

The following result can be trivially seen.

%\begin{lemma}
%\label{lemma10}
%A vector $\mathbf{v} \in \mathbb{R}_{\geq 0}^r$ satisfies $\mathbf{v} \in \mathcal{F}_\epsilon(J)$ if and only if
%\[\Gamma_J \mathbf{v} \leq 0.\] 
%\end{lemma}

%\begin{lemma}
%\label{lemma11}
%A vector $\mathbf{v} \in \mathbb{R}_{\geq 0}^r$ satisfies $\mathbf{v} \in \mathcal{C}(I)$ if and only if
%\[\Gamma_I \mathbf{v} \leq 0.\] 
%\end{lemma}

\begin{lemma}
\label{lemma12}
The condition $\mathcal{C}(I) \cap \mathcal{F}_\epsilon(J) \subseteq$ ker$(I,J,\epsilon)$ is satisfied if and only if
\[\tilde{\Gamma} \mathbf{v} \leq \mathbf{0} \mbox{ for } \mathbf{v} \in \mathbb{R}_{\geq 0}^{r} \; \; \; \; \; \Longrightarrow \; \; \; \; \; \mathbf{v} \in \emph{ker}(\tilde{\Gamma}).\] 
\end{lemma}

\subsection{Facets and Non-Critical Semi-Locking Sets}
\label{facetssection}

In general, determining whether a semi-locking set $I$ is weakly dynamically non-emptiable can be tedious. To this end, in this section we show that there are classes of semi-locking sets which are necessarily weakly dynamically nonemptiable: namely, semi-locking sets corresponding to facets (i.e. sets $F_I$ of dimension $s-1$) of a weakly reversible mechanism, and semi-locking sets which are non-critical.

Facets are the central topic of consideration in \cite{A-S}, where the authors prove the following result.

\begin{theorem}[Theorem 3.4, \cite{A-S}]
\label{andersonshiu}
Consider a weakly reversible mass-action system with bounded trajectories. Suppose that every semi-locking set $I$ is such that $F_I$ is a facet or empty. Then the system is persistent. 
\end{theorem}

The authors also connect the notion of a facet with the traditional notion of dynamical non-emptiability (Corollary 3.5, \cite{A-S}). Their result, however, overstates the implications of $F_I$ being a facet. It can be shown that semi-locking sets $I$ corresponding to facets $F_I$ may fail to be dynamically non-emptiable if there is a reaction $\mathcal{R}_i$ such that: (1) $\mathcal{R}_i \not\curlyeqprec_I \mathcal{R}_j$ for all $\mathcal{R}_j$ in the same linkage class (a connected portion of the reaction graph); and (2) $\mathcal{R}_i$ produces no stoichiometric change in the species in $I$. (See Example 1 of Section \ref{examplessection}.)

We now generalize this result by showing that there is no such exemption for weak dynamical non-emptiability.

\begin{theorem}
\label{facettheorem}
Consider a weakly reversible mass-action system with a semi-locking set $I \subseteq \left\{ 1, \ldots, m \right\}$. If $F_I$ is a facet then $I$ is weakly dynamically non-emptiable.
\end{theorem}

\begin{proof}
We will follow closely the proofs of Theorem 3.2 and Corollary 3.5 contained in \cite{A-S}.

In their proof for Theorem 3.2, the authors show that there exist $z_j>0$, $j \in I$, and $\gamma_i \in \mathbb{R}$, $i \in \left\{ 1, \ldots, r \right\}$, such that
\begin{equation}
\label{29291}
\gamma_i z_j = \beta_{ij}-\alpha_{ij}
\end{equation}
for all $j \in I$ and $i \in \left\{ 1, \ldots, r \right\}$. In other words, relative to the support of $I$, every reaction vector (the columns of $\Gamma$) lies within the span of a single vector which is strictly positive on the support of $I$. We will let $\mathbf{z} \in \mathbb{R}_{>0}^{n_I}$ denote the vector with the elements $z_j$, $j \in I$, indexed in order.

It follows immediately from (\ref{29291}) that every reaction in the system contributes either: (1) a net gain to all species in $I$, (2) a net loss to all species in $I$, or (3) no stoichiometric change to species in $I$. We can also divide the reactions according to the linkage classes $\mathcal{L}_k$, $k = 1, \ldots, \ell$. By weak reversibility each linkage class is strongly connected. We will let $R^{(k)}$ denote the reactions in the $k^{th}$ linkage class, and $R_+^{(k)}$, $R_-^{(k)},$ and $R_0^{(k)}$ denote respectively the reactions in the $k^{th}$ linkage class which contribute a net gain, a net loss, or no change to all species in $I$. We notice that $\gamma_i > 0$ for $i \in R_+^{(k)}$, $\gamma_i < 0$ for $i \in R_-^{(k)}$, and $\gamma_i = 0$ for $i \in R_0^{(k)}$. Combined with (\ref{29291}), this division gives
\begin{equation}
\label{8389}
\Gamma_I \mathbf{v} = \mathbf{z} \left[ \sum_{k=1}^\ell \left( \sum_{i \in R_+^{(k)}} \gamma_i v_i - \sum_{j \in R_-^{(k)}} |\gamma_j| v_j \right) \right].
\end{equation}

We now proceed to construct our set $J \subseteq \mathcal{R}_I$ according to (\ref{J}). Since the system is weakly reversible, it follows that the product complex of every reaction $\mathcal{R}_i$ is itself a reactant complex for some other reaction (which we will denote $\mathcal{R}_{i'}$) in the $k^{th}$ linkage class. Since reactions may only produce simultaneous gain or loss to the species in $I$, it follows that: (1) $\mathcal{R}_{i'} \curlyeqprec_I \mathcal{R}_i$ if $i \in R_+^{(k)}$, (2) $\mathcal{R}_i \curlyeqprec_I \mathcal{R}_{i'}$ if $i \in R_-^{(k)}$, and (3) $\mathcal{R}_{i'} \not\curlyeqprec_I \mathcal{R}_i$ and $\mathcal{R}_i \not\curlyeqprec_I \mathcal{R}_{i'}$ if $i \in R_0^{(k)}$. (Notice here that $\mathcal{R}_{i'} \not\curlyeqprec_I \mathcal{R}_i$ and $\mathcal{R}_{i} \not\curlyeqprec_I \mathcal{R}_{i'}$ for $i \in R^{(k)}$ implies $\alpha_{ij} = \alpha_{{i'}j}$ for all $j \in I$, although this does not hold for general systems.)

Since the ordering relationship is transitive and can be extended throughout each linkage class by weak reversibility, it follows that the set of reactions corresponding to each linkage class $\mathcal{L}_k$ either: (1) contributes no stoichiometric change to the system (i.e. $i \in R_0^{(k)}$ for all reactions $\mathcal{R}_i$ corresponding to reactions in $\mathcal{L}_k$), or (2) contains a reaction $\mathcal{R}_{i_k}$, $i_k \in R_+^{(k)}$, such that $\mathcal{R}_j \curlyeqprec_I \mathcal{R}_{i_k}$ for all $j \in R_-^{(k)}$. We will ignore linkage classes included in case (1) since they do not affect (\ref{8389}).

We define the set
\[J = \mathop{\bigcup_{k = 1}^\ell}_{j \in R_-^{(k)}} (j,i_k).\]
Now assume that $\Gamma_J \mathbf{v} \leq 0$. This implies that we have
\begin{eqnarray*}
\Gamma_I \mathbf{v} & = & \mathbf{z} \left[ \sum_{k=1}^\ell \left( \sum_{i \in R_+^{(k)} \setminus i_k} \gamma_i v_i + \gamma_{i_k} v_{i_k} - \sum_{j \in R_-^{(k)}} |\gamma_j| v_j \right) \right] \\ & \geq & \mathbf{z} \left[ \sum_{k=1}^\ell \left( \sum_{i \in R_+^{(k)} \setminus i_k} \gamma_i v_i + \left( \gamma_{i_k} - \epsilon \sum_{j \in R_-^{(k)}} | \gamma_j | \right) v_{i_k} \right) \right].
\end{eqnarray*}
Regardless of the values of $\gamma_j$, $j \in R_-^{(k)}$, and $\gamma_{i_k}>0$ for $k=1, \ldots, \ell$, we can pick an $\epsilon > 0$ sufficient small so that $\Gamma_I \mathbf{v}\geq \mathbf{0}$ for every $\mathbf{v} \in \mathbb{R}_{\geq 0}^r$. In order to satisfy $\Gamma_I \mathbf{v} \leq \mathbf{0}$ for $\mathbf{v} \in \mathbb{R}_{\geq 0}^r$, therefore, we require $v_i = 0$ for all $i \in R_+^{(k)}$ and $i \in R_-^{(k)},$ $k = 1, \ldots, \ell$. We notice that $v_i > 0$ is permitted for $i \in R_0^{(k)}$; however, neither $\Gamma_I$ nor $\Gamma_J$ contain nonzero entries in their columns corresponding to the elements in $i \in R_0^{(k)}$. It follows that $\tilde{\Gamma} \mathbf{v} \leq 0$ for $\mathbf{v} \in \mathbb{R}_{\geq 0}$ entails $\mathbf{v} \in$ ker$(\tilde{\Gamma})$. Since this is a sufficient condition for the weak dynamical non-emptiability of $I$ by Lemma \ref{lemma12}, we are done.
\end{proof}

In \cite{A3}, the authors divide semi-locking sets according to whether they are critical or not. They first handle the case of non-critical semi-locking sets which culminates in Theorem \ref{sontag1}. It is only in the discussion of critical semi-locking sets that they introduce the further condition of dynamical nonemptiability. Here we simplify this discussion by showing that every non-critical semi-locking set is weakly dynamically non-emptiable and therefore falls within the scope of the discussion in Section \ref{criticalsemilockingsetssection}.

\begin{theorem}
\label{noncritical}
Consider a semi-locking set $I \subseteq \left\{ 1, \ldots, m \right\}$. If $I$ is non-critical then it is weakly dynamically non-emptiable.
\end{theorem}

\begin{proof}
Suppose a semi-locking set $I \subseteq \left\{ 1, \ldots, m \right\}$ is non-critical. This means that $I$ corresponds to the support of a semi-conservation vector so that there exists a $\mathbf{c} \in \mathbb{R}_{\geq 0}^m$ satisfying
\[\mathbf{c} = \left\{ \begin{array}{ll} c_i > 0, \hspace{0.3in} & i \in I \\ c_i = 0, & i \not\in I\end{array} \right.\]
such that
\begin{equation}
\label{938}
\mathbf{c}^T \Gamma = \mathbf{0}^T.
\end{equation}
It follows from (\ref{938}) that there exists a $\mathbf{y} \in \mathbb{R}_{>0}^{n_I}$ such that
\[\mathbf{y}^T \Gamma_I = \mathbf{0}^T.\]

Since this implies condition \emph{2.} of Farkas' Lemma is satisfied, it follows that Condition \emph{1.} must necessarily be violated. It follows that any $\mathbf{v} \in \mathbb{R}_{\geq 0}^r$ satisfying $\Gamma_I \mathbf{v} \leq 0$ must be such that $\mathbf{v} \in$ ker$(\Gamma_I)$. By Lemma \ref{lemma12}, however, this is the condition for weak dynamical nonemptiability taking $J = \emptyset$ which is sufficient to prove the result.
\end{proof}

\subsection{Main Persistence Result}
\label{criticalsemilockingsetssection}

We are now prepared to present the main result of this paper. The following result is a generalization of Theorem \ref{sontag2} and includes Theorem \ref{sontag1} for mass-action kinetics by Theorem \ref{noncritical}.

\begin{theorem}
\label{sontag4}
Consider a mass-action system with bounded solutions. Suppose that every semi-locking set is weakly dynamically non-emptiable. Then the system is persistent.
\end{theorem}

\begin{proof}
We know by Theorem \ref{bigtheorem} that a system with bounded solutions is persistent if, for every semi-locking set $I$, there is an $\alpha$ satisfying
\[\alpha = \left\{ \begin{array}{l} \alpha_i < 0, \mbox{ for } i \in I \\ \alpha_i = 0, \mbox{ for } i \not\in I \end{array} \right.\]
so that, for every compact subset $K$ of $L_I$, there exists a neighbourhood $U$ of $K$ in $\mathbb{R}_{\geq 0}^m$ such that $\langle \alpha, \mathbf{f}(\mathbf{x}) \rangle \leq 0$ for all $\mathbf{x} \in U$.

By assumption, every critical semi-locking set is weakly dynamically non-emptiable, which means that there exists an $\epsilon > 0$ and a $J \subseteq \mathcal{R}_I$ such that $\mathcal{C}(I) \cap \mathcal{F}_\epsilon(J) \subseteq$ ker$(I,J,\epsilon)$. By Lemma \ref{lemma12}, this implies that
\[\tilde{\Gamma} \mathbf{v} \leq \mathbf{0}, \mbox{ for } \mathbf{v} \in \mathbb{R}_{\geq 0}^{r} \; \; \; \; \; \Longrightarrow \; \; \; \; \; \mathbf{v} \in \mbox{ker}(\tilde{\Gamma}).\]
It follows that condition \emph{1.} of Lemma \ref{farkaslemma} is not satisfied. Consequently, in order to satisfy \emph{2.}, there must exist a $\mathbf{c} \in \mathbb{R}_{>0}^{n_I + n_J}$ such that $\mathbf{c}^T \tilde{\Gamma} \geq \mathbf{0}^T$.

We partition $\mathbf{c} \in \mathbb{R}_{>0}^{n_I + n_J}$ so that
\[\mathbf{c} = \left[ \begin{array}{c} \mathbf{c}_I \\ \mathbf{c}_J \end{array} \right]\]
 where $\mathbf{c}_I \in \mathbb{R}_{>0}^{n_I}$ and $\mathbf{c}_J \in \mathbb{R}_{>0}^{n_J}$. From this it follows that
\begin{equation}
\label{29321}
\mathbf{c}^T \tilde{\Gamma} = \mathbf{c}_I^T \Gamma_I + \mathbf{c}_J^T \Gamma_J \geq \mathbf{0}^T.
\end{equation}
Multiplying through the right-hand side of (\ref{29321}) by $R(\mathbf{x})$, we have
\begin{equation}
\label{important}\mathbf{c}_I^T \Gamma_I R(\mathbf{x}) + \mathbf{c}_J^T \Gamma_J R(\mathbf{x}) = -\langle \alpha, \mathbf{f}(\mathbf{x}) \rangle + \mathbf{c}_J^T \Gamma_J R(\mathbf{x}) \geq 0
\end{equation}
where $\alpha \in \mathbb{R}_{\leq 0}^m$ is the vector $-\mathbf{c}_I$ extended over the support $I$ and has zeroes elsewhere. Clearly $\alpha$ satisfies
\[\alpha = \left\{ \begin{array}{ll} \alpha_i < 0 \hspace{0.2in} & \mbox{for } i \in I\\ \alpha_i = 0 & \mbox{for } i \not\in I. \end{array} \right.\]

By Lemma \ref{lemma4}, for every compact subset $K$ of $L_I$ and every $\epsilon > 0$, there exists a neighbourhood $U$ of $K$ in $\mathbb{R}_{\geq 0}^m$ such that $\Gamma_J R(\mathbf{x}) \leq \mathbf{0}$. It follows from (\ref{important}) that
\[\langle \alpha, \mathbf{f}(\mathbf{x}) \rangle \leq \mathbf{c}_J^T \Gamma_J R(\mathbf{x}) \leq 0\]
for all $\mathbf{x} \in U$.

Since this holds for every semi-locking set $I$ by assumption, it follows by Theorem \ref{bigtheorem} that the system is persistent.
\end{proof}

There are several points worth emphasizing about Theorem \ref{sontag4} as it contrasts with Theorem \ref{sontag2}. In our result the requirement that the system be conservative has been replaced by the more general assumption that solutions are bounded, and we do not require the assumption that there are no nested critical locking sets. Since a system being conservative implies solutions are bounded, the first is not a significant change; however, we have opened the result to non-conservative systems for which solutions can be bounded by another method, as is the case with complex balanced systems (see Section \ref{complexbalancedsystemssection}). In Section \ref{examplessection} we will see examples where persistence holds despite the systems not being conservative.

We have removed the distinction between critical and non-critical semi-locking sets. We do not need to make this distinction since every non-critical semi-locking set is weakly dynamically nonemptiable by Theorem \ref{noncritical} and therefore trivially included in Theorem \ref{sontag4}.

It is also worth noting that Theorem \ref{sontag4} holds even if the $\alpha_{ij}$ in (\ref{gamma}) differ from the mass-action exponents in (\ref{reactionvector}) provided some regularity conditions hold. Specifically, persistence still applies if, rather than (\ref{reactionvector}), we consider the reaction vector $R(\mathbf{x}) \in \mathbb{R}_{>0}^r$ with entries
\[R_i(\mathbf{x}) = k_i \prod_{j=1}^m x_j^{\tilde{\alpha}_{ij}}\]
where $\tilde{\alpha}_{ij} > 0$ if and only if $\alpha_{ij} > 0$ and the modified feasibility cone $\mathcal{F}_\epsilon(J)$ is redefined to apply to the partial ordering $\mathcal{R}_i \curlyeqprec_I \mathcal{R}_j$ with the $\tilde{\alpha}_{ij}$'s rather than the $\alpha_{ij}$.

\subsection{Complex Balanced Systems}
\label{complexbalancedsystemssection}

Persistence is of particular interest in the study of the \emph{complex balanced systems} first introduced in \cite{F1,H,H-J1}. It is beyond the scope of this paper to develop complex balanced systems in any detail except to note that they are known to have a unique complex balanced equilibrium concentration within each compatibility class and that this concentration is locally asymptotically stable relative to the compatibility class (Lemma 4C and Theorem 6A, \cite{H-J1}).

Despite significant work, however, the following conjecture remains unproven. (This conjecture was first presented in \cite{H1}. We state the conjecture here as it is stated in \cite{C-D-S-S}.)

\begin{proposition}[Global Attractor Conjecture]
\label{globalattractorconjecture}
For any complex balanced system and any starting point $\mathbf{x}_0 \in \mathbb{R}_{>0}^m$, the associated complex balanced equilibrium point $\mathbf{x}^*$ of $\mathsf{C}_{\mathbf{x}_0}$ is a global attractor of $\mathsf{C}_{\mathbf{x}_0}$.
\end{proposition}

Although no general proof is known, many limitations on solutions not tending toward $\mathbf{x}^*$ have been found. Theorem 3.2 of \cite{S-M} guarantees that the $\omega$-limit set of a complex balanced system consists either of the unique positive equilibrium $\mathbf{x}^*$ in $\mathsf{C}_{\mathbf{x}_0}$ or of complex balanced equilibria lying on $\partial \mathbb{R}_{>0}^m$; consequently, persistence of complex balanced systems suffices to affirm Proposition \ref{globalattractorconjecture}. Furthermore, since $\partial \mathbb{R}_{>0}^m$ decomposes into the sets $L_I$ and $\omega(\mathbf{x}_0) \cap L_I \not= \emptyset$ for $\mathbf{x}_0 \in \mathbb{R}_{>0}^m$ implies $I$ is a semi-locking set by Lemma \ref{lemma461}, it follows that we need only prove $\omega(\mathbf{x}_0) \cap L_I = \emptyset$ for the sets $L_I$ corresponding to semi-locking sets.

Most recent research on the \emph{Global Attractor Conjecture} has made use of these restrictions (\cite{A,A-S,A3,C-D-S-S,S-J}). For a summary of the major results of this research to date, see the discussion preceding Theorem 4.6 of \cite{A-S}. We will append to this result the implications of Theorem \ref{sontag4}; however, we begin with a Lemma. (In the following we let $F_I$ be associated with the initial condition $\mathbf{x}_0 \in \mathbb{R}_{>0}^m$ and $\tilde{F}_I$ be associated with $\tilde{\mathbf{x}}_0 \in \mathbb{R}_{>0}^m$.)

\begin{lemma}
\label{facelemma}
Consider a chemical reaction network. Consider a set $I \subseteq \left\{ 1, \ldots, m \right\}$ and suppose that $\tilde{F}_I \not= \emptyset$ for some $\tilde{\mathbf{x}}_0 \in \mathbb{R}_{>0}^m$. Then, for every $\mathbf{x}_0 \in \mathbb{R}_{>0}^m$, either dim$(F_I)=$dim$(\tilde{F}_I)$ or $F_I=\emptyset$. Furthermore, for every $\mathbf{x} \in L_I$, there exists $\mathbf{x}_0 \in \mathbb{R}_{>0}^m$ such that $\mathbf{x} \in F_I$.
\end{lemma}

\begin{proof}
Consider the set $\tilde{L}_I = \left\{ \mathbf{x} \in \mathbb{R} \; | \; x_i = 0 \mbox{ if } i \in I \right\}$. It is clear that $L_I$ is relatively interior to $\tilde{L}_I$, i.e. $\forall$ $\mathbf{x} \in L_I$, $\exists$ $\epsilon > 0$ such that $B_\epsilon (\mathbf{x}) \cap \tilde{L}_I \subseteq L_I$. ($B_\epsilon(\mathbf{x})$ is the standard Euclidean ball of radius $\epsilon$ centered at $\mathbf{x}$.) Now consider the affine space $(\mathbf{x}_0+S) \cap \tilde{L}_I$ and suppose $F_I = (\mathbf{x}_0 + S) \cap L_I \not= \emptyset$. Then, $\forall$ $\mathbf{x} \in F_I$, $\exists$ $\epsilon > 0$ such that $B_\epsilon(\mathbf{x}) \cap [(\mathbf{x}_0 + S) \cap \tilde{L}_I] = (\mathbf{x}_0 + S) \cap [ B_\epsilon(\mathbf{x}) \cap \tilde{L}_I ] \subseteq (\mathbf{x}_0 + S) \cap L_I = F_I$. Consequently, $F_I$ is relatively interior to $(\mathbf{x}_0 + S) \cap \tilde{L}_I$. Since the dimension of $(\mathbf{x}_0 + S) \cap \tilde{L}_I$ is the same for all $\mathbf{x}_0 \in \mathbb{R}_{>0}^m$, it follows that dim$(F_I)$ is the same for all $\mathbf{x}_0 \in \mathbb{R}_{>0}^m$ so long as $F_I \not= \emptyset$. This proves the first claim.

Since $\tilde{F}_I \not= \emptyset$ by assumption, we can consider an arbitrary $\tilde{\mathbf{x}} \in \tilde{F}_I$. By definition, we have that $(\tilde{\mathbf{x}}_0 - \tilde{\mathbf{x}})_i > 0$ for $i \in I$ and $\tilde{\mathbf{x}}_0 - \tilde{\mathbf{x}} \in S$. Now choose an arbitrary $\mathbf{x} \in L_I$. It follows from the definition of $L_I$ that $\mathbf{x}_0 = \mathbf{x} + \epsilon (\tilde{\mathbf{x}}_0 - \tilde{\mathbf{x}}) \in \mathbb{R}_{>0}^m$ for $\epsilon > 0$ sufficiently small. Since $\mathbf{x} \in F_I$ for $\mathbf{x}_0 \in \mathbb{R}_{>0}^m$ and $\mathbf{x} \in L_I$ was chosen arbitrarily, the second claim follows.
\end{proof}

This result guarantees that if $\tilde{F}_I$ is a facet (or vertex) for some $\tilde{\mathbf{x}}_0 \in \mathbb{R}_{>0}^m$, then $F_I$ is a facet (or vertex) for any $\mathbf{x}_0 \in \mathbb{R}_{>0}^m$ so long as $F_I \not= \emptyset$. Furthermore, it guarantees that $L_I$ can be completely partitioned into sets $F_I$ corresponding to facets (or vertices).

We are now prepared to prove the following application of Theorem \ref{sontag4} to complex balanced systems. It should be noted that, while facets are weakly dynamically non-emptiable by Theorem \ref{facettheorem}, no comparable result holds for vertices (consider the origin in Example 2 of Section \ref{examplessection}). Consequently, the following result cannot be attained as a simple application of Theorem \ref{sontag4}.

\begin{corollary}
\label{sontag5}
Consider a complex balanced mass-action system. Suppose that every set $F_I$ corresponding to a semi-locking set $I$ is either a facet, a vertex, or empty, or that $I$ is weakly dynamically non-emptiable. Then the \emph{Global Attractor Conjecture} holds for this system.
\end{corollary}

\begin{proof}
For complex balanced systems, $\omega(\mathbf{x}_0) \cap F_I = \emptyset$ for every $F_I$ corresponding to a vertex (Proposition 20 of \cite{C-D-S-S}) or the empty set (trivially). Also, from Corollary 3.3 of \cite{A-S}, we have that $\omega(\mathbf{x}_0) \cap F_I \not= \emptyset$ implies $\omega(\mathbf{x}_0) \cap \partial F_I \not= \emptyset$ for all $F_I$ corresponding to facets; however, since $\partial F_I$ corresponds to some $F_{\tilde{I}}$ not corresponding to a facet, this is a contradiction. It follows from Lemma \ref{facelemma} that $\omega(\mathbf{x}_0) \cap L_I = \emptyset$ for any semi-locking set $I$ such that $F_I$ is a facet or a vertex for some $\mathbf{x}_0 \in \mathbb{R}_{>0}^m$. It remains to show that $\omega(\mathbf{x}_0) \cap L_I = \emptyset$ for every $L_I$ corresponding to a weakly dynamically non-emptiable semi-locking set $I$.

By Theorem \ref{sontag4} we know that for each semi-locking set $I$ which is weakly dynamically non-emptiable, there is an $\alpha$ satisfying
\[\alpha = \left\{ \begin{array}{l} \alpha_i < 0, \mbox{ for } i \in I \\ \alpha_i = 0, \mbox{ for } i \not\in I \end{array} \right.\]
so that, for every compact subset $K$ of $L_I$, there exists a neighbourhood $U$ of $K$ in $\mathbb{R}_{\geq 0}^m$ such that $\langle \alpha, \mathbf{f}(\mathbf{x}) \rangle \leq 0$ for all $\mathbf{x} \in U$. Since complex balanced systems are bounded, we are justified in using the inductive hypothesis of Theorem 3.3 of \cite{S-J} from $|I| = m$ to $|I| = 1$ to conclude that, for all $\mathbf{x}_0 \in \mathbb{R}_{>0}^m$, $\omega(\mathbf{x}_0) \cap \partial \mathbb{R}_{>0}^m = \emptyset$. It follows by Theorem 3.2 of \cite{S-M} that the \emph{Global Attractor Conjecture} holds for trajectories of such a system, and we are done.
\end{proof}

\section{Examples}
\label{examplessection}

In this section, we present three chemical reaction systems which illustrate how Theorem \ref{facettheorem}, Theorem \ref{sontag4}, and Corollary \ref{sontag5} work and what their limitations are.

The first two are examples of non-conservative systems where we exploit the results of Section 3. We show that the first example contains a semi-locking set which is weakly dynamically non-emptiable but not dynamically non-emptiable in the sense introduced in \cite{A3}. The third is an example of a system which does not fall within the bounds of the results discussed in this paper or any other papers of which the authors are aware.

All of the systems considered in this section are complex balanced for all sets of rate constants and consequently fall within the scope of the discussion in Section \ref{complexbalancedsystemssection}. Since the semi-locking set $I=\left\{ 1, \ldots, m \right\}$ corresponds to $F_I = \left\{ \mathbf{0} \right\}$, which is always a vertex of any compatibility class it is in, we may exclude it when considering complex balanced systems, since no trajectory may approach it. We will call any semi-locking set $I=\left\{ 1, \ldots, m \right\}$ \emph{trivial}. (For further discussion of sufficient conditions to determine complex balancing, see \cite{F1} and \cite{H}.)\\

\textbf{Example 1:} Consider the mass-action system
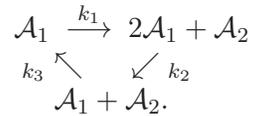
\[\begin{array}{lll} \mathcal{A}_1 \; \stackrel{k_1}{\longrightarrow} \; 2 \mathcal{A}_1 + \mathcal{A}_2 \\ \; {}_{k_3} \nwarrow \; \; \; \; \; \; \swarrow {}_{k_2} \\ \; \; \; \; \; \mathcal{A}_1 + \mathcal{A}_2. \end{array}\]
For this system, we have
\[\Gamma = \left[ \begin{array}{rrr} 1 & -1 & 0 \\ 1 & 0 & -1 \end{array} \right] \hspace{0.2in} \mbox{ and } \hspace{0.2in} R(\mathbf{x}) = \left[ \begin{array}{c} k_1 x_1 \\ k_2 x_1^2 x_2 \\ k_3 x_1 x_2 \end{array} \right]\]
and the system is governed by $\dot{\mathbf{x}} = \Gamma R(\mathbf{x})$.

We notice first of all that the system is not conservative and therefore does not fall within the scope of the systems considered in \cite{A3}. We might still be tempted to ask whether the system has semi-locking sets which are dynamically non-emptiable, so we consider the semi-locking set $I = \left\{ 1 \right\}$. Relative to this set, we have $\mathcal{R}_2 \curlyeqprec_{I} \mathcal{R}_1$ and $\mathcal{R}_2 \curlyeqprec_{I} \mathcal{R}_3$ so that $\mathcal{F}_\epsilon(I) \cap \mathcal{C}(I) = \left\{ \mathbf{0} \right\}$ corresponds to finding a $\mathbf{v} \in \mathbb{R}_{\geq 0}^3$ such that $v_1 - v_2 \leq 0$, $v_2 \leq \epsilon v_1$, and $v_2 \leq \epsilon v_3$. This can clearly be satisfied for any $\mathbf{v} = \left[ 0 \; \; 0 \; \; v_3 \right]^T$ where $v_3 \geq 0$. Since $\mathcal{F}_\epsilon(I) \cap \mathcal{C}(I) \not= \left\{ \mathbf{0} \right\}$, it follows that the system contains a critical semi-locking set which is not dynamically non-emptiable.

We notice, however, that $F_I$ is a facet of $\mathsf{C}_{\mathbf{x}_0}=\mathbb{R}_{>0}^2$ since $s=2$ and dim$(F_I)=1$. It follows from Theorem \ref{facettheorem} that $I$ is weakly dynamically non-emptiable. Since the system is complex balanced for all sets of rate constants and $I$ is the only non-trivial semi-locking set, the \emph{Global Attractor Conjecture} holds for this system by Corollary \ref{sontag5}. (This result could also be attained by application of Theorem 4.6 of \cite{A-S}, although it should be pointed out that $F_I$ is an example of a facet which is not dynamically non-emptiable in the traditional sense so that Corollary 3.5 of the same paper cannot be applied.)\\

\textbf{Example 2:} Consider the system
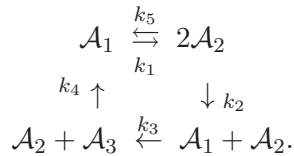
\[\begin{array}{c} \mathcal{A}_1 \; \; \displaystyle{\stackrel{k_5}{\mathop{\leftrightarrows}_{k_1}}} \; \; 2\mathcal{A}_2 \\ {}^{k_4} \uparrow \hspace{0.5in} \downarrow {}_{k_2} \\ \mathcal{A}_2 + \mathcal{A}_3 \; \stackrel{k_3}{\leftarrow} \; \mathcal{A}_1 + \mathcal{A}_2. \end{array}\]
The system is governed by the dynamics $\dot{\mathbf{x}} = \Gamma R(\mathbf{x})$ where
\[\Gamma = \left[ \begin{array}{rrrrr} -1 & 1 & -1 & 1 & 1 \\ 2 & -1 & 0 & -1 & -2 \\ 0 & 0 & 1 & -1 & 0 \end{array} \right] \hspace{0.2in} \mbox{ and } \hspace{0.2in} R(\mathbf{x}) = \left[ \begin{array}{c} k_1 x_1 \\ k_2 x_2^2 \\ k_3 x_1 x_2 \\ k_4 x_2 x_3 \\ k_5 x_2^2 \end{array} \right].\]

This example was first considered in \cite{S-J}, where the authors showed that the system is non-conservative, complex balanced for all sets of rate constants, and has only the non-trivial semi-locking set $I = \left\{ 1, 2 \right\}$. By the methodology presented in that paper, however, they could not find an $\alpha$ corresponding to $I$ satisfying (\ref{condition1}) and (\ref{condition2}). Since the system is not conservative, the results of \cite{A3} cannot be applied, and since $I$ is not a facet, the results of \cite{A-S} cannot be applied. Here we will show that such an $\alpha$ does in fact exist by showing that $I$ is weakly dynamically non-emptiable.

We have that
\[\Gamma_{I} = \left[ \begin{array}{rrrrr} -1 & 1 & -1 & 1 & 1 \\ 2 & -1 & 0 & -1 & -2 \end{array} \right].\]
We have $\mathcal{R}_2 \curlyeqprec_{I} \mathcal{R}_4$, $\mathcal{R}_3 \curlyeqprec_{I} \mathcal{R}_1$, $\mathcal{R}_3 \curlyeqprec_{I} \mathcal{R}_4$, and $\mathcal{R}_5 \curlyeqprec_{I} \mathcal{R}_4$ so that
\begin{equation}
\label{J2}
\mathcal{R}_I = \left\{ (2,4), (3,1), (3,4), (5,4) \right\}.
\end{equation}
We pick the subset $J = \left\{ (3,4) \right\}$ so that
\[\Gamma_J = \left[ \begin{array}{rrrrr} 0 & 0 & 1 & -\epsilon & 0 \end{array} \right]. \]

The condition $\tilde{\Gamma} \mathbf{v} \leq \mathbf{0}$ for $\mathbf{v} \in \mathbb{R}_{\geq 0}^5$ is equivalent to the system $-v_1+v_2-v_3+v_4+v_5 \leq 0$, $2v_1-v_2-v_4-2v_5 \leq 0$, and $v_3-\epsilon v_4 \leq 0$ for $v_i \geq 0$, $i=1, \ldots, 5$. Taking a positive linear combination of these conditions yields $v_2 + (1-2\epsilon)v_4 \leq 0$. For $0 < \epsilon < 1/2$, this can be satisfied for $v_2 \geq 0$ and $v_4 \geq 0$ if and only if $v_2=v_4=0$. It then follows from the third condition that $v_3=0$. The remaining conditions can be satisfied so long as $v_1 = v_5 \geq 0$ so that
\[\mathbf{v} \in \mbox{span} \left\{ \left[ \begin{array}{ccccc} 1 & 0 & 0 & 0 & 1 \end{array} \right]^T \right\} \subseteq \mbox{ker}(\tilde{\Gamma}).\]
By Lemma \ref{lemma12}, the semi-locking set $I$ is weakly dynamically non-emptiable. Since trajectories are bounded by virtue of the system being complex balanced, it follows from Theorem \ref{sontag4} that the system is persistent and from Corollary \ref{sontag5} that it satisfies Proposition \ref{globalattractorconjecture}.

In order to illustrate how the machinery of this result really works, we will complete the analysis for $I$ up to the point of applying Theorem \ref{bigtheorem}. From Lemma \ref{farkaslemma} we have that there exists a $\mathbf{c} \in \mathbb{R}_{>0}^3$ such that $\mathbf{c}^T \tilde{\Gamma} \geq \mathbf{0}^T$; in fact, we can find it explicitly. This is satisfied if we choose $c_1=2$, $c_2=1$, $c_3=2$, and $0 < \epsilon < 1$, for which values we have
\[\begin{split} \mathbf{c}^T \tilde{\Gamma} R(\mathbf{x}) & = \left[ \begin{array}{rrr} 2 & 1 & 2 \end{array} \right] \left[ \begin{array}{rrrrr} -1 & 1 & -1 & 1 & 1 \\ 2 & -1 & 0 & -1 & -2 \\ 0 & 0 & 1 & -\epsilon & 0 \end{array} \right] \left[ \begin{array}{c} k_1 x_1 \\ k_2 x_2^2 \\ k_3 x_1 x_2 \\ k_4 x_2 x_3 \\ k_5 x_2^2 \end{array} \right] \\ & = -\alpha^T \dot{\mathbf{x}} + 2(k_3 x_1 x_2 - \epsilon k_4 x_2 x_3) \geq 0 \end{split}\]
where $\alpha = [-2 \; \; -1  \; \; 0]^T$. It follows that $\alpha^T \dot{\mathbf{x}} \leq 2(k_3 x_1 x_2 - \epsilon k_4 x_2 x_3) \leq 0$ in a neighbourhood of any compact subset of $F_{I}$ since $k_3 x_1 x_2 \leq \epsilon k_4 x_2 x_3$ under the same conditions by Lemma \ref{lemma4}. This is exactly the condition which was expected for application of Theorem \ref{bigtheorem}, which completes the connection with Theorem \ref{sontag4}.

It is worth reemphasizing that not all sets $J$ satisfying (\ref{J2}) are sufficient to show that $I$ is weakly dynamically non-emptiable. For instance, if we had selected $\tilde{J} = \left\{ (2,4), (3,1), (3,4) \right\}$, we would have had
\[\Gamma_{\tilde{J}} = \left[ \begin{array}{rrrrr} 0 & 1 & 0 & -\epsilon & 0 \\ -\epsilon & 0 & 1 & 0 & 0 \\ 0 & 0 & 1 & -\epsilon & 0 \end{array} \right].\]
In this case, we can satisfy $\tilde{\Gamma} \mathbf{v} \leq \mathbf{0}$ by choosing
\[\mathbf{v} \in \mbox{span} \left\{ \left[ \begin{array}{rrrrr} 1 & 0 & 0 & 0 & 1 \end{array} \right]^T \right\}\]
but ker$(\tilde{\Gamma}) = \left\{ \mathbf{0} \right\}$. Consequently, $\tilde{J}$ is insufficient to show that $I$ is weakly dynamically non-emptiable.

It is also worth noting that $J$ is not the only choice sufficient for showing $I$ is weakly dynamically non-emptiable. In fact, the maximal set $\tilde{J} = \mathcal{R}_I$ works with ker$(\tilde{\Gamma}) = \left\{ \mathbf{0} \right\}$. (In other words, $I$ is dynamically non-emptiable in the sense introduced in \cite{A3}! We remain unable to use Theorem 4 of \cite{A3}, however, because this system is not conservative.) We can see also that it is easier to demonstrate weak dynamical non-emptiability with some choices of $J$ than with others, an advantage which would become even more pronounced for larger systems.\\

\textbf{Example 3:} Now consider the system
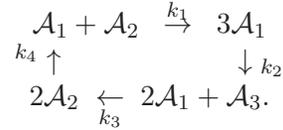
\[\begin{array}{c} \mathcal{A}_1 + \mathcal{A}_2 \; \; \stackrel{k_1}{\rightarrow} \; \; 3\mathcal{A}_1 \\ {}^{k_4} \uparrow \; \; \; \; \; \; \; \; \; \; \; \; \; \; \; \; \; \; \; \; \; \; \downarrow {}_{k_2} \\ 2\mathcal{A}_2 \; \; \displaystyle{\mathop{\leftarrow}_{k_3}} \; \; 2\mathcal{A}_1 + \mathcal{A}_3. \end{array}\]
The system is governed by the dynamics $\dot{\mathbf{x}} = \Gamma R(\mathbf{x})$ where
\[\Gamma = \left[ \begin{array}{rrrr} 2 & -1 & -2 & 1 \\ -1 & 0 & 2 & -1 \\ 0 & 1 & -1 & 0 \end{array} \right] \hspace{0.2in} \mbox{ and } \hspace{0.2in} R(\mathbf{x}) = \left[ \begin{array}{c} k_1 x_1 x_2 \\ k_2 x_1^3 \\ k_3 x_1^2 x_3 \\ k_4 x_2^2 \end{array} \right].\]
The system is non-conservative, complex balanced for all sets of rate constants, and has only the non-trivial semi-locking set $I = \left\{ 1, 2 \right\}$. The system is not conservative, so the results of \cite{A3} cannot be applied, and $I$ is not a facet, so the results of \cite{A-S} cannot be applied. We consider whether $I$ is weakly dynamically non-emptiable.

We have only the condition $\mathcal{R}_2 \curlyeqprec_{I} \mathcal{R}_3$ so that $J \subseteq \left\{ (2,3) \right\}$. Choosing the maximal such set we have\[\tilde{\Gamma} = \left[ \begin{array}{rrrr} 2 & -1 & -2 & 1 \\ -1 & 0 & 2 & -1 \\ 0 & 1 & -\epsilon & 0 \end{array} \right].\]
It is clear that $\mathbf{v} = \left[ 0 \; \; 0 \; \; 1 \; \; 2 \right]^T$ satisfies $\tilde{\Gamma} \mathbf{v} \leq 0$ but
\[\mathbf{v} \not\in\mbox{ ker}(\tilde{\Gamma}) =\mbox{ span}\left\{ \left[ - \epsilon \; \; - \epsilon \; \; -1 \; \;  -2 + \epsilon \right]^T \right\}\]
for any $\epsilon > 0$. Since the condition $\tilde{\Gamma} \mathbf{v} \leq 0$ for $\mathbf{v} \in \mathbb{R}_{\geq 0}^r$ does not imply $\mathbf{v} \in$ ker$(\tilde{\Gamma})$ for the trivial set $J = \emptyset$ either, it follows that $I$ is not weakly dynamically non-emptiable and thus the results of this paper cannot be applied.

The only other approach that we know of to handle such a situation is Corollary 3.15 of \cite{S-J}. It can be checked, however, that there are eight \emph{strata} which intersect $L_I$ and that the corresponding vectors $\sum_{j=1}^k \mathbf{s}_{\mu_i(j)}$, $i=1, \ldots, \delta$, $k=1, \ldots, l_i-1$, do not have a common $\alpha$ satisfying (\ref{condition1}) and (\ref{condition2}) such that either Condition 1 or Condition 2 is satisfied. We submit, therefore, that this is an example of a system whose persistence lies beyond the scope of known theory.

\addtocontents{toc}{\protect\contentsline{chapter}{\vspace{0.15in} \textbf{Bibliography}}{}}

\end{document}